\documentclass[a4paper,12pt]{article}
\usepackage{geometry}
\usepackage{amsmath, amsthm, amssymb}
\usepackage{makeidx}
\usepackage{thmtools}
\usepackage{anysize}
\usepackage{verbatim}
\usepackage{authblk}
\usepackage{hyperref}
\makeindex

\newtheoremstyle{custom}
{12pt} {12pt} {} {} {\bfseries} {:} {.5em} {}
\theoremstyle{custom}
\newtheorem*{prob}{Problem}
\newtheorem*{thrm}{Theorem}

\newcounter{count}
\declaretheorem[numberlike=count]{Theorem}
\declaretheorem[numberlike=count]{Lemma}
\declaretheorem[numberlike=count]{Corollary}
\declaretheorem[numberlike=count]{Proposition}
\newcommand{\lr}[1]{\left(#1\right)}

\newcommand{\NN}[0]{\mathbb N}
\newcommand{\ZZ}[0]{\mathbb Z}
\newcommand{\CC}[0]{\mathbb C}

\newcommand{\FF}[0]{\mathbb F}

\newcommand{\one}[0]{\textbf{1}}
\renewcommand{\hat}[1]{\widehat{#1}}
\renewcommand{\bar}[1]{\overline{#1}}
\newcommand{\eps}[0]{\varepsilon}

\newcommand{\disc}[0]{\text{disc}}
\renewcommand{\mod}[0]{\text{ mod }}

\makeindex
\begin{document}


\title{Capturing Forms in Dense Subsets of Finite Fields}

\author{Brandon Hanson\thanks{The author is supported in part by NSERC of
Canada.}\\
  University of Toronto\\
  \texttt{bhanson@math.toronto.edu}
}

\date{}
\maketitle

\begin{abstract}
An open problem of arithmetic Ramsey theory asks if given a finite $r$-colouring
$c:\NN\to\{1,\ldots,r\}$ of the natural numbers, there exist $x,y\in \NN$ such
that $c(xy)=c(x+y)$ apart from the trivial solution $x=y=2$. More generally, one
could replace $x+y$ with a binary linear form and $xy$ with a binary quadratic
form. In this paper we examine the analogous problem in a finite field $\FF_q$. 
Specifically, given a linear form $L$ and a quadratic from $Q$ in two variables, 
we provide estimates on the necessary size of $A\subset \FF_q$ to guarantee that
$L(x,y)$ and $Q(x,y)$ are elements of $A$ for some $x,y\in\FF_q$.
\end{abstract}
\section{Introduction}
In this paper we consider a finite field analog of the following open problem in
arithmetic Ramsey theory \cite{Hindman}.
\begin{prob}
For any $r$-colouring
$c:\NN\to \{1,\ldots,r\}$ of the natural numbers, is it possible to solve
$c(x+y)=c(xy)$ apart from the trivial solution $(x,y)=(2,2)$?
\end{prob}
One might suspect that in fact a stronger result might hold, namely that any
sufficiently dense set of natural numbers contains the elements $x+y$ and $xy$ for some $x$
and $y$. This would immediately solve the problem since one of the colours
in any finite colouring must be sufficiently dense. Such a result is impossible
however, since the odd numbers provide a counter example and are fairly dense
in many senses of the word. Fortunately, this simple parity obstruction
disappears in the finite field setting. Indeed, in \cite{Shkredov}, the
following was proved.\footnote{The author would like to thank J.
Solymosi for bringing this result to his attention.}

\begin{thrm}
Let $p$ be a prime number, and $A_1, A_2, A_3\subset \FF_p$ be any sets,
$|A_1||A_2||A_3| \geq 40p^\frac{5}{2}$. Then there are $x, y\in\FF_p$ such that
$x + y \in A_1$, $xy \in A_2$ and $x \in A_3$.
\end{thrm}

Now, let $q=p^n$ be an odd prime power and $\FF_q$ a finite field of order $q$.
Given a binary linear form $L(X,Y)$ and a binary quadratic form $Q(X,Y)$, define
$N_q(L,Q)$ to be the smallest integer $k$ such that for any subset $A\subset
\FF_q$ with $|A|\geq k,$ there exists $(x,y)\in\FF_q^2$ with $L(x,y),Q(x,y)\in
A$. In this paper we give estimates on the size of $N_q(L,Q)$. Namely, we will
prove the following theorem.

\begin{Theorem}
\label{thm1}
Let $\FF_q$ be a finite field of odd order. Let $Q\in\FF_q[X,Y]$ be a
binary quadratic form with non-zero discriminant and let
$L\in\FF_q[X,Y]$ be a binary linear form not dividing $Q$. Then we have
\[\log q\ll N_q(L,Q)\ll\sqrt q.\]
\end{Theorem} 

This theorem is the content of the next two sections. In the final section, we
provide remarks on the analogous problem in the ring of integers modulo $N$ when
$N$ is composite.

\section{Upper Bounds}
Let $L(X,Y)$ be a linear form and $Q(X,Y)$ be a quadratic form, both with
coefficients in $\FF_q$. Suppose $A$ is an arbitrary subset of
$\FF_q$. We will reduce the problem of solving
$L(x,y),Q(x,y)\in A$ to estimating a character sum.

By a multiplicative character, we mean a group homomorphism
$\chi:\FF_q^\times\to\CC^\times$. We say $\chi$ is non-trivial if it is not
constant, ie.
$\chi\not\equiv 1$. We also extend such characters to $\FF_q$ with the
convention that $\chi(0)=0$. One of the most useful features of characters is
that for $\chi$ non-trivial, we have \[\sum_{x\in\FF_q}\chi(x)=0.\] 
The quadratic character on $\FF_q$
is the character given by \[\chi(c)=
\begin{cases}
1 &\text{ if }c\neq 0\text{ is a square}\\
-1 &\text{ if }c\neq 0\text{ is not a square}\\
0 &\text{ if }c=0.
\end{cases}\]

\begin{Lemma}
\label{lem1}
Let $Q\in\FF_q[X,Y]$ be a binary
quadratic form and let $L\in\FF_q[X,Y]$ be a binary
linear form. Suppose $a,b\in\FF_q$. Then there exist $r,s,t\in\FF_q$ depending
only on $L$ and $Q$ such that \[|\{(x,y)\in\FF_q^2:L(x,y)=a\text{ and
}Q(x,y)=b\}|=|\{y\in\FF_q:ry^2+say+ta^2=b\}|.\] Furthermore, $r=0$ if and only
if $L|Q$ and $r=s=0$ if and only if $L^2|Q$.
\end{Lemma}
\begin{proof}
Write $L(X,Y)=a_1X+a_2Y$ where without loss of
generality we can assume $a_1\neq 0$. We can
factor \[Q(X,Y)=tL(X,Y)^2+sL(X,Y)Y+rY^2.\] If $L(x,y)=a$ then we obtain
\[Q(x,y)=ta^2+say+ry^2.\] The $y^2$ coefficient vanishes if and only if
$Q=LM$ for some linear form $M$. The $y$ and $y^2$
coefficients vanish if and only if $Q=tL^2$. Certainly, any solution to
$L(x,y)=a$ and $Q(x,y)=b$ gives a solution $y$ of $ry^2+say+ta^2=b$. Conversely,
if $y$ is such a solution, setting $x=a_1^{-1}(a-a_2y)$ produces a solution $(x,y)$.
\end{proof}

\begin{Corollary}
\label{cor1}
Let $Q\in\FF_q[X,Y]$ be a binary
quadratic form and let $L\in\FF_q[X,Y]$ be a binary
linear form not dividing $Q$. For $a,b\in\FF_q$, the number of solutions to
$L(x,y)=a$ and $Q(x,y)=b$ is \[1+\chi((s^2-4rt)a^2+4rb)\] where $\chi$
is the quadratic character.
\end{Corollary}
\begin{proof}
The quantity $(sa)^2-4r(ta^2-b)$ is the discriminant of $ry^2+say+ta^2-b$. The
result follows from the definition of $\chi$ and the quadratic formula.
\end{proof}

In fact, from \autoref{lem1}, we can essentially handle the situation when
$L|Q$.

\begin{Corollary}
Let $Q\in\FF_q[X,Y]$ be a binary
quadratic form and let $L\in\FF_q[X,Y]$ be a binary
linear form dividing $Q$. Then $N_q(L,Q)=1$ if $L^2$ does not divide $Q$,
otherwise $N_q(L,Q)\geq\frac{q+1}{2}$.
\end{Corollary}
\begin{proof}
Let $A\subset\FF_q$. The number of pairs $(x,y)$ with $L(x,y),Q(x,y)\in A$
is \[\sum_{x,y}\one_A(L(x,y))\one_A(Q(x,y))=\sum_{a\in
A}\sum_{y\in\FF_q}\one_A(say+ta^2)\] by the above lemma. If $sa\neq 0$ then
$say+ta^2$ ranges over $\FF_q$ as $y$, and the inner sum is
$|A|$. In this case there are in fact $|A|^2$ solutions $(x,y)$. If $a=0$ then $0\in A$
and we can take $(x,y)=(0,0)$. If $s=0$ then the sum is $q\sum_{a\in
A}\one_A(a^2t)$. If we set \[A=\begin{cases}t\cdot N=\{tn:n\in N\} & \text{ if
}t\neq 0\\N & \text{ if }t=0\end{cases}\] where $N$ is the set of non-squares in $\FF_q$, then there are no solutions.
This shows that $N_q(L,Q)\geq\frac{q+1}{2}$.
\end{proof}

We now handle the case that $L$ does not divide $Q$. The
following estimate is essentially due to Vinogradov (see for instance the
excercises of chapter 6 in \cite{V} for the analogous result for exponentials).

\begin{Lemma}
\label{vinogradov}
Let $A,B\subset \FF_q$ and suppose $\chi$ is a non-trivial
multiplicative character. Then if $u,v\in\FF_q^\times$ \[\sum_{a\in A}\sum_{b\in
B}\chi(ua^2+vb)\leq2\sqrt{q|A||B|}.\]
\end{Lemma}
\begin{proof}
Let $S$ denote the sum in question. Then \[|S|\leq\sum_{b\in
B}\left|\sum_{a\in A}\chi(ua^2+vb)\right|\leq |B|^\frac{1}{2}\lr{\sum_{b\in
\FF_q}\left|\sum_{a\in A}\chi(ua^2+vb)\right|^2}^\frac{1}{2}\] by
Cauchy's inequality.
Expanding the sum in the second factor, we get 
\begin{eqnarray*}
\sum_{a_1,a_2\in
A}\sum_{\substack{b\in\FF_q\\ua_2^2+vb\neq
0}}\chi\lr{\frac{ua_1^2+vb}{ua_2^2+vb}}&=&\sum_{a_1,a_2\in
A}\sum_{\substack{b\in\FF_q\\ua_2^2+vb\neq
0}}\chi\lr{1+\frac{u(a_1^2-a_2^2)}{ua_2^2+vb}}\\
&=&\sum_{a_1,a_2\in
A}\sum_{b\in\FF_q^\times}\chi\lr{1+u(a_1^2-a_2^2)b}
\end{eqnarray*} after the change of
variables $(ua_2^2+vb)^{-1}\mapsto b$. When $a_1^2\neq a_2^2$, the values of
$1+u(a_1^2-a_2^2)b$ range over all values of $\FF_p$ save $1$ as $b$
traverses $\FF_q^\times$. Hence, in this case, the sum amounts to $-1$. It
follows that the total is at most $4q|A|$.
\end{proof}

Recall that the discriminant of a quadratic form $Q(X,Y)=b_1X^2+b_2XY+b_3Y^2$ is
defined to be $\disc(Q)=b_2^2-4b_1b_3$.

\begin{Proposition}
\label{prop1}
Let $Q\in\FF_q[X,Y]$ be a binary
quadratic form and let $L\in\FF_q[X,Y]$ be a binary
linear form not dividing $Q$. Then
$N_q(L,Q)\leq 2\sqrt q+1$ if $\disc(Q)\neq 0$ otherwise $N_q(L,Q)\geq
\frac{q-1}{2}$.
\end{Proposition}
\begin{proof}
Let $A\subset \FF_q$. By \autoref{cor1}, the number of pairs $(x,y)$ with $L(x,y),Q(x,y)\in A$
is \[\sum_{x,y}\one_A(L(x,y))\one_A(Q(x,y))=\sum_{a,b\in
A}1+\chi(Da^2+4rb)\] where
$D=s^2-4rt$. One can check that in fact $D=a_1^{-2}\disc(Q)$. 

If $D=0$ then $\chi(Da^2+4rb)+1=\chi(r)\chi(b)+1$. This will
be indentically zero if $A$ is chosen to be the squares or non-squares according
to the value of $\chi(r)$. Hence, if $\disc(Q)=0$ then $N_q(L,Q)\geq
\frac{q-1}{2}$.

Now assume $D\neq 0$. Summing over $a,b\in A$ the number of solutions is
\[|A|^2+\sum_{a,b\in A}\chi(Da^2+4rb)=|A|^2+E(A).\] By \autoref{vinogradov},
$E(A)<|A|^2$ when $|A|\geq2\sqrt q+1$ and the result follows.
\end{proof}

In the case that $A$ has particularly nice structure, we can improve the upper
bound. Suppose $q=p$ is prime and $A$ is an interval. Then as above the
number of pairs $(x,y)$ with $L(x,y),Q(x,y)\in A$ is \[|A|^2+\sum_{a,b\in
A}\chi(Da^2+4rb).\] Now \[\sum_{a,b\in A}\chi(Da^2+4rb)\leq \sum_{a\in
A}\left|\sum_{b\in A}\chi(Da^2/4r+b)\right|.\] A well-known result of Burgess
states that the inner sum (which is also over an interval) is $o(|A|)$ whenever
$|A|\gg p^{\frac{1}{4}+\eps}$ (see \cite{IK}, chapter 12).

\section{A Lower Bound}
In this section we give a lower bound for $N_q(L,Q)$ in the case that $L$ does
not divide $Q$ and $\disc(Q)\neq 0$. To do so we need to produce a set $A$ such
that $L(x,y)$ and $Q(x,y)$ are never both elements of $A$. Equivalently, we need
to produce a set $A$ for which $\chi(Da^2+4rb)=-1$ for all pairs $(a,b)\in
A\times A$.

Let $a\in\FF_q$ and define \[X_a(b)=\begin{cases} 1&\text{if
}\chi(Da^2+4rb)=\chi(Db^2+4ra)=-1\\
0 & \text{otherwise.}\end{cases}\]

Thus the desired set $A$ will have $X_a(b)=1$ for $a,b\in A$. The idea behind
our argument is probabilistic. Suppose we create a graph $\Gamma$ with vertex
set \[V=\{a\in\FF_q:X_a(a)=1\}\] and edge set \[E=\{\{a,b\}:X_a(b)=X_b(a)=1\}.\]
These edges appear to be randomly distributed and occur with probability roughly
$\frac{1}{4}$. In this setting, $N_q(L,Q)$ is one more than the clique number of
$\Gamma$ (ie. the size of the largest complete subgraph of $\Gamma$). Let
$G(n,\delta)$ be the graph $n$ vertices that is the result of connecting two
vertices randomly and independently with probability $\delta$.
Such a graph has clique number roughly $\log n$ (see \cite{AS}, chapter 10). One
is tempted to treat $\Gamma$ as such a graph and construct a clique by greedily
choosing vertices, and indeed this is how the set $A$ is constructed. It is
worth mentioning that this model suggests that the right upper bound for
$N_q(L,Q)$ is also roughly $\log n$.

\begin{Lemma}
Let $B\subset\FF_q$. Then for $a\in \FF_q$, we have \[\sum_{b\in
B}X_a(b)=\frac{1}{4}\sum_{b\in B}(1-\chi(Da^2+4rb))(1-\chi(Db^2+4ra))+O(1).\]
\end{Lemma}
\begin{proof}
The summands on the right are 
\[(1-\chi(Da^2+4rb))(1-\chi(Db^2+4ra))=\begin{cases}
4 &\text{if }\chi(Da^2+4rb)=\chi(Db^2+4ra)=-1\\
2 &\text{if }\{\chi(Da^2+4rb),\chi(Db^2+4ra)\}=\{0,-1\}\\
1 &\text{if }\chi(Da^2+4rb)=\chi(Db^2+4ra)=0\\
0 &\text{otherwise.}\\
\end{cases}\]
For fixed $a$, the second and third cases can only occur for $O(1)$ values of
$b$.
\end{proof}

We will use the following well-known theorem of Weil, see for instance chapter
11 of \cite{IK}.

\begin{Theorem}[Weil] Suppose $\chi\in\hat{\FF_q^\times}$ has order $d>1$ and
$f\in\FF_q[X]$ is not of the form $f=g^d$ for some $g\in\bar{\FF_q}[X]$. If $f$
has $m$ distinct roots in $\bar{\FF_q}$ then \[\left|\sum_{x\in
\FF_q}\chi(f(x))\right|\leq m\sqrt q.\]
\end{Theorem}

\begin{Proposition}
Let $A,B\subset\FF_q$ with $|A|,|B|>\sqrt q$. Then \[\sum_{a\in
A}\sum_{b\in B}X_a(b)=\frac{|A||B|}{4}+O(|A||B|^\frac{1}{2}q^\frac{1}{4}).\]
\end{Proposition}
\begin{proof}
By the preceding lemma, it suffices to estimate
\begin{eqnarray*}
&&\sum_{a\in A}\frac{1}{4}\left(\sum_{b\in
B}(1-\chi(Da^2+4rb))(1-\chi(Db^2+4ra))\right)+O(1)\\
&=&\frac{|A||B|}{4}-\frac{1}{4}\sum_{a\in A}\sum_{b\in B}\chi(Da^2+4rb)-\frac{1}{4}\sum_{a\in A}\sum_{b\in B}\chi(Db^2+4ra)\\
&&+\frac{1}{4}\sum_{a\in A}\sum_{b\in B}\chi((Da^2+4rb)(Db^2+4ra))+O(|A|).
\end{eqnarray*}
By Lemma 1 of the previous section, the first two sums above are
$O(\sqrt{q|A||B|})=O(|A||B|^\frac{1}{2}q^\frac{1}{4})$. By Cauchy's inequality,
the final sum is bounded by
\[|B|^\frac{1}{2}\lr{\sum_{b\in\FF_q}\left|\sum_{a\in
A}\chi((Da^2+4rb)(Db^2+4ra))\right|^2}^\frac{1}{2}.\] Expanding the square
modulus, the second factor is the square-root of \[\sum_{a_1,a_2\in
A}\sum_{b\in\FF_q}\chi((Da_1^2+4rb)(Db^2+4ra_1)(Da_2^2+4rb)(Db^2+4ra_2)).\] By Weil's theorem, 
the inner sum is bounded by $6\sqrt q$ when the polynomial \[f(b)=(Da_1^2+4rb)(Db^2+4ra_1)(Da_2^2+4rb)(Db^2+4ra_2)\] 
is not a square. This happens for all but $O(|A|)$ pairs $(a_1,a_2)$. Hence the bound is
$O(|A|q+|A|^2\sqrt q)$. Since $|A|>\sqrt q$, this is $O(|A|^2\sqrt q)$ and the overall bound is
$O(|A||B|^\frac{1}{2}q^\frac{1}{4})$.
\end{proof}

We immediately deduce the following.

\begin{Corollary}
There is an absolute constant $c>0$ such that if
$B\subset\FF_q$ with $|B|\geq c\sqrt q$ then there is an element $a\in B$ such that \[|\{b\in
B:X_a(b)=1\}|\geq \frac{1}{8}|B|.\]
\end{Corollary}
\begin{proof}
Indeed, taking $A=B$ in the preceeding theorem, \[\max_{a\in
B}\left\{\sum_{b\in B}X_a(b)\right\}\geq \frac{1}{|B|}\sum_{a,b\in
B}X_a(b)=\frac{|B|}{4}+O(q^\frac{1}{4}|B|^\frac{1}{2})\geq \frac{|B|}{8}\] when
$|B|>c\sqrt q$ for some appropriately chosen $c$.
\end{proof}

\begin{Proposition}
Let $Q\in\FF_q[X,Y]$ be a binary
quadratic form and let $L\in\FF_q[X,Y]$ be a binary
linear form not dividing $Q$. Then if $\disc(Q)\neq 0$
we have $N_q(L,Q)\gg\log q$.
\end{Proposition}
\begin{proof}
We will construct a clique in the graph $\Gamma$ introduced
above. First we claim that $|V|=\frac{q-1}{2}+O(1)$.
Indeed \[\sum_{a\in\FF_q^\times}\chi(Da^2+4ra)=\sum_{a\in\FF_q^\times}\chi(a^{-2})\chi(Da^2+4ra)=\sum_{a\in\FF_q^\times}\chi(D+4ra^{-1})=O(1)\]
by orthogonality. The final term is $O(1)$ and the claim follows since $\chi$ takes on the values
$\pm 1$ on $\FF_q^\times$. 

Now set $V_0=V$ and assume $q$ is large. Write $|V_0|=c'q>c\sqrt q$ (with $c$ as
in the preceeding corollary and $c'\approx \frac{1}{2}$). For $a\in V_0$, let
$N(a)$ denote the neighbours of $a$ (ie. those $b$ which are joined to $a$ by an
edge). Then there is an $a_1\in V_0$ such that $|N(a_1)|\geq c'q/8$. Let
$A_1=\{a_1\}$, let $V_1=N(a_1)\subset V_0$, and for $a\in V_1$ let
$N_1(a)=N(a)\cap V_1$. By choice, all elements of $V_1$ are connected to $a_1$.
Now $|V_1\setminus A_1|\geq c'q/8-1\geq c'q/16$ so, provided this is at least
$c'q/16$, there is some element $a_2$ of $V_1\setminus A_1$ such that
$|N_1(a_2)|\geq |V_1\setminus A_1|/8$. Let $A_2=A_1\cup\{a_2\}$,
$V_2=N_1(a_2)\subset V_1$ and define $N_2(a)=N(a)\cap V_2$. Once again each
element of $V_2$ is connected to each element of $A_2$. We repeat this process
provided that at stage $i$ there exists an element $a_{i+1}\in V_i\setminus A_i$
with $|N_i(a_{i+1})|\geq |V_i\setminus A_i|/8$. We set
$A_{i+1}=A_i\cup\{a_{i+1}\}$ and observe that $A_{i+1}$ induces a clique. We may
iterate provided $|V_i\setminus A_i|> c\sqrt q$ which is guaranteed for $i\ll
\log q$. The final set $A_i$ (which has size $i$) will be the desired set $A$.
\end{proof}

The combination of this proposition and \autoref{prop1} completes the proof of
\autoref{thm1}.
\section{Remarks for Composite Modulus}

Consider the analogous question in the ring $\ZZ/N\ZZ$ with $N$ odd. Let
$L(X,Y)=a_1X+a_2Y$ with $(a_1,N)=1$ and $Q(X,Y)=b_1X^2+b_2XY+b_3Y^2$.
We then let $A\subset\ZZ/N\ZZ$ and wish to find $(x,y)\in(\ZZ/N\ZZ)^2$ such that
$L(x,y),Q(x,y)\in A$. As before, this amounts to finding a solution to
\[Q(a_1^{-1}(a-a_2Y),Y)=b\] for some $a,b\in A$. In general, one cannot to find
a solution based on the size of $A$ alone unless $A$ is very large. Indeed, if
$p$ is a small prime dividing $N$ and $t\mod p$ is chosen such that the
discriminant of \[Q(a_1^{-1}(t-a_2Y),Y)-t\] is a non-residue modulo $p$ then
taking $A=\{a\mod N:a\equiv t\mod p\}$ provides a set of density $1/p$ which
fails admit a solution.

\section*{Acknowledgements}

I would like to thank Leo Goldmakher and John Friedlander for introducing me to
the problem and helpful discussion.

\bibliographystyle{plain}
\bibliography{bibliography}
\end{document}